\documentclass[12pt, reqno]{amsart}
\usepackage{amsmath, amsthm, amscd, amsfonts, amssymb, graphicx, color}

\newtheorem{df}{Definition}[section]
\newtheorem{thm}[df]{Theorem}
\newtheorem{pro}[df]{Proposition}

\begin{document}
\setcounter{page}{1}

\title[Semi-Browder Joint Spectra]{Tensor Products And The \\ 
Semi-Browder Joint Spectra}

\author{Enrico Boasso}

\begin{abstract} Given two complex Banach spaces $X_1$ and $X_2$, 
a tensor product of $X_1$ and $X_2$, $X_1\tilde{\otimes}X_2$, 
in the sense of J. Eschmeier ([5]), and two
finite tuples of commuting operators, $S=(S_1,\ldots ,S_n)$ and 
$T=(T_1,\ldots ,T_m)$, defined on $X_1$ and $X_2$ respectively,
we consider the $(n+m)$-tuple of operators defined on 
$X_1\tilde{\otimes}X_2$, $(S\otimes I,I\otimes T)=
(S_1\otimes I,\ldots ,S_n\otimes I,I\otimes T_1,\ldots ,I
\otimes T_m)$, and 
we give a description of the semi-Browder joint spectra introduced by 
V. Kordula, V. M\"uller and V. Rako$\check{c}$evi$\acute{ c}$ in [7]
and of the split semi-Browder joint spectra (see section 3), of the
$(n+m)$-tuple $(S\otimes I ,I\otimes T)$, in terms of the corresponding
joint spectra of $S$ and $T$. This result is in some sense a
generalization of a formula obtained for other various Browder 
spectra in Hilbert spaces and for tensor products of operators and for tuples of  
the form $(S\otimes I ,I\otimes T)$. In addition, we also describe all the
mentioned joint spectra for a tuple of left and right multiplications
defined on an
operator ideal between Banach spaces in the sense of [5]. \end{abstract}

\maketitle

\section{Introduction}
\noindent Given a complex Banach space $X$, V. Kordula, V. M\"uller and V. Rako$\check{c}$evi$\acute{ c}$ 
extended in [7] the notion of upper and lower semi-Browder spectrum of an operator
to $n$-tuples 
of commuting operators, and they proved the main spectral properties for this joint spectra,
i.e., the compactness, nonemptiness, the projection property and the spectral mapping property.\par

\indent On the other hand, there are other many joint Browder spectra, 
for example, we may consider the one introduced by R. E. Curto and A. T. Dash in [2], 
$\sigma_{b}$, and the joint Browder spectra defined by A. T. Dash in [3], $\sigma_{b}^1$, $\sigma_{b}^2$ and $\sigma_{b}^T$. 
By the observation which follows Definition 4 in [3] and the Example in [7],
we have that the Browder spectra of V. Kordula, V. M\"uller and V. Rako$\check{c}$evi$\acute{ c}$,
$\sigma_{{\mathcal B}_{+}}$ and $\sigma_{{\mathcal B}_{-}}$,
differ, in general, from the other mentioned joint Browder spectra. However, if we consider
two complex Hilbert spaces $H_1$ and $H_2$, and $S$ and $T$ two operators 
defined on $H_1$ and $H_2$ respectively, by [2] and [3, Theorem 7] we have that the joint 
Browder spectra $\sigma_{b}$, $\sigma_{b}^1$, $\sigma_{b}^2$ and $\sigma_{b}^T$
of the tuple of operators $(S\otimes I,I\otimes T)$ defined on $H_1\overline{\otimes} H_2$, coincide with the set
$$
\sigma_{b}(S)\times \sigma(T)\cup\sigma(S)\times\sigma_{b}(T),
$$
where $\sigma$ and $\sigma_{b}$ denote, respectively, the usual and the Browder spectrum of
an operator. \par
\indent Moreover, if $S=(S_1 ,\ldots ,S_n)$, respectively $T=(T_1 ,\ldots ,T_m)$,
is an $n$-tuple, respectively an $m$-tuple, of commuting operators defined on the Hilbert space $H_1$, 
respectively $H_2$, R. E. Curto and A. T. Dash computed in [2] the Browder spectum
of the $(n+m)$-tuple $(S\otimes I,I\otimes T)=(S_1\otimes I,\ldots ,S_n\otimes I,I\otimes T_1,\ldots ,I\otimes T_m)$,
and they obtained the formula
$$
\sigma_{b}(S\otimes I,I\otimes T)=\sigma_b(S)\times\sigma_T(T)\cup\sigma_T(T)\times\sigma_b(T),
$$
where $\sigma_T$ denotes the Taylor joint spectrum (see [9]).\par

\indent In this article we give in some sense a generalization of the above
formulas for commutative tuples of Banach spaces operators and for the semi-Browder joint
spectra. Indeed, we
consider two complex Banach spaces, $X_1$
and $X_2$, a tensor product between
$X_1$ and $X_2$ in the sense of J. Eschmeier ([5]) $X_1\tilde{\otimes} X_2$, $S$ and $T$, two
commuting tuples of Banach space operators defined on $X_1$ and $X_2$ respectively,
and we describe the semi-Browder
joint spectra introduced in [7], $\sigma_{{\mathcal B}_{+}}$ and 
$\sigma_{{\mathcal B}_{-}}$, and the split semi-Browder joint spectra 
$sp_{{\mathcal B}_{+}}$ and $sp_{{\mathcal B}_{-}}$ (see section 3) of
the tuple $(S\otimes I,I\otimes T)$, in terms of the 
corresponding semi-Browder joint spectra and 
of the defect and the approximate point spectra of $S$ and $T$.  
The results that we have obtained extend in same way the above formulas,
see section 5.
Furthermore, since for our objective we need to know the Fredholm
joint spectra of J.J. Buoni, R. Harte and T. Wickstead of $(S\otimes I,I\otimes T)$ 
([1]) and its split versions ([4]) we also descrive in section 4 these joint 
spectra. \par

\indent In addition, by similar arguments we describe in section 6 all the mentioned joint
spectra for a tuple of left and right multiplications defined on an operator ideal
between Banach spaces in the sense of [5].\par

\indent However, in order to give our descriptions, we need to
introduce the split semi-Browder joint spectra of a tuple of 
commuting Banach space operators, and to prove their main spectral
properties (see section 3).\par
 
\indent The article is organized as follows. In section 2 we recall several
definitions and results which we need for our work. In section 3 we 
introduce the split semi-Browder joint spectra and prove their main
spectral properties. In section 4 we
compute the semi-Fredholm joint spectra of $(S\otimes I,I\otimes T)$.
In section 5 we compute the semi-Browder joint spectra of $(S\otimes I,I\otimes T)$,
and in section 6, the semi-Fredholm and the semi-Browder joint  
spectra of a tuple of left and right multiplications defined on an operator ideal between Banach spaces in the 
sense of [5].\par

\section{Preliminaries}
\noindent Let us begin our work by recalling the definitions of the lower 
semi-Fredholm and of the lower semi-Browder joint spectra of a finite tuple of 
operators denfined on a complex Banach space, for a complete exposition see [1] and [7].\par

\indent Let $T=(T_1,\ldots ,T_n)$ be an $n$-tuple of commuting operators defined on 
a Banach space $X$, and for $k\in \mathbb{N}$ define $M_k(T)=R(T_1^k)+\ldots +R(T_n^k)$. 
Clearly $X\supseteq M_1(T)\supseteq M_2(T)\supseteq\ldots\supseteq M_k(T)\supseteq\ldots $
Let us set $R^{\infty}(T)=\cap_{k=1}^{\infty}M_k(T)$. We now may recall the
definition of the lower semi-Browder joint spectrum (see [7]).\par

We say that $T=(T_1,\ldots ,T_n)$ is \it lower semi-Browder \rm if 
codim $R^{\infty}(T)<\infty$. The set of all lower semi-Browder $n$-tuples 
is denoted by ${\mathcal B}_{-}^{(n)}(X)$, and the \it lower semi-Browder spectrum \rm is 
the set
$$
\sigma_{{\mathcal B}_{-}}(T)=\{\lambda\in\mathbb{C}^n:\hbox{ } T-\lambda\notin {\mathcal B}_{-}^{(n)}(X)\},
$$
where $\lambda=(\lambda_1,\ldots ,\lambda_n)$ and $T-\lambda=
(T_1-\lambda_1 I,\ldots ,T_n-\lambda_n I)$. 
 
\indent As usual (see [1]), we say that $T=(T_1,\ldots ,T_n)$ is \it lower semi-Fredholm, \rm
i.e., $T\in\Phi_{-}^{(n)}(X)$, if 
$$
\hbox{\rm codim } M_1(T)=\hbox{\rm codim } (R(T_1)+\ldots + R(T_n))<\infty,
$$
equivalently, if the operator $\hat{T}\colon X^n\to X$ defined by
$\hat{T}(x_1,\ldots ,x_n)=T_1(x_1)+\ldots +T_n(x_n)$ is lower semi-Fredholm,
i.e., $R(\hat{T})$ is closed and has finite codimension. The lower semi-Fredholm spectrum is the set
$$
\sigma_{\Phi_{-}}(T)=\{\lambda\in\mathbb{C}^n:\hbox{ } T-\lambda\notin \Phi_{-}^{(n)}(X)\}.
$$

\indent An easy calculation shows that
$$
\sigma_{\Phi_{-}}(T)\subseteq\sigma_{{\mathcal B}_{-}}(T)\subseteq\sigma_{\delta}(T),
$$
where $\sigma_{\delta}(T)$ is the defect spectrum of $T$, i.e.,
$$
\sigma_{\delta}(T)=\{\lambda\in\mathbb{C}^n:\hbox{ } \hbox{ codim }M_1(T-\lambda)\ne 0\}.
$$
\indent Moreover, it is easy to see that the lower semi-Browder spectrum may be decomposed as
the disjoint union of two sets, 
$$
\sigma_{{\mathcal B}_{-}}(T)=\sigma_{\Phi_{-}}(T)\cup {\mathcal A(T)},
$$
where 
$$
{\mathcal A(T)}=\{\lambda\in \mathbb{C}^n:\hbox{} \forall k\in\mathbb{N}, 1\le \hbox{\rm codim }M_k(T-\lambda)<\infty,
\hbox{\rm codim }M_k(T-\lambda)\xrightarrow[k\to\infty]{ }\infty\}.
$$

\indent Now, we recall the definition of the upper semi-Fredholm and the upper 
semi-Browder joint spectra, as above, for a complete exposition see [1] and [7].\par

\indent If $T$ is an $n$-tuple of commuting operators defined on a Banach space
$X$, then $T$ is said \it upper semi-Fredholm, \rm  i.e., $T\in\Phi_{+}^{(n)}(X)$, if the map
$\tilde{T}\colon X\to X^n$ defined by $\tilde{T}(x)=(T_1(x),\ldots ,T_n(x))$
is upper semi-Fredholm,
equivalently, if $\tilde{T}$ has finite dimensional null space and closed range. Moreover, $T$ is said \it upper semi-Browder, \rm  i.e., $T\in{\mathcal B}_{+}^{(n)}(X)$, 
if $T\in \Phi_{+}^{(n)}(X)$ and
$\dim N^{\infty}(T)<\infty$, where
$$
N^{\infty}(T)=\cup_{k\in\mathbb{N}}[N(T_1^k)\cap\ldots\cap N(T_n^k) ].
$$

\indent As above,  the \it upper semi-Fredholm spectrum  \rm is the set
$$
\sigma_{\Phi_{+}}(T)=\{\lambda\in\mathbb{C}^n: T-\lambda\notin \Phi_{+}^{(n)}(X)\}, 
$$
and the \it upper semi-Browder spectrum \rm is the set
$$
\sigma_{{\mathcal B }_{+}}(T)=\{\lambda\in\mathbb{C}^n: T-\lambda\notin {\mathcal B}_{+}^{(n)}(X)\}. 
$$

\indent In addition, it is easy to see that
$$
\sigma_{\Phi_{+}}(T)\subseteq\sigma_{{\mathcal B}_{+}}(T)\subseteq\sigma_{\pi}(T),
$$
where $\sigma_{\pi }(T)$ denotes the approximate point spectrum of $T$,
$$
\sigma_{\pi}(T)=\{\lambda\in\mathbb{C}^n:\hbox{ } N(\tilde{T-\lambda})\ne 0\, \hbox{ or } R(\tilde{T-\lambda}) 
\hbox{ is not closed }\}.
$$

\indent Moreover, it is easy to see that the upper semi-Browder spectrum may be decomposed as
the disjoint union of two sets, 
$$
\sigma_{{\mathcal B}_{+}}(T)=\sigma_{\Phi_{+}}(T)\cup {\mathcal D(T)},
$$
where 
$
{\mathcal D(T)}=\{\lambda\in \mathbb{C}^n:\forall k\in\mathbb{N}, 1\le \dim N_k(\tilde{T-\lambda})<\infty,\hbox{ } R(\tilde{T-\lambda})\hbox{ is}$ 
$\hbox{closed, and }\dim N_k(\tilde{T-\lambda})\xrightarrow[k\to\infty]{ }\infty \}
$, where 
$N_k(\tilde{T-\lambda})=N(\tilde{(T-\lambda)}^k)$ and $(T-\lambda)^k=((T_1-\lambda_1)^k,\ldots , (T_n-\lambda_n)^k)$.\par

\indent Let us recall that the semi-Fredholm and the semi-Browder joint spectra are compact nonempty subsets of $\mathbb{C}^n$,
which also satisfy the projection property and the analytic spectral 
mapping theorem for tuples of holomorphic functions
defined on a neighborhood of the Taylor joint spectrum [9] (see [4] and [7]).\par 

\indent On the other hand, in order to prove our main results, we have to recall the axiomatic tensor product between 
Banach spaces introduced by J. Eschmeier in [5]. This notion will be central in 
this work. For a complete exposition see [5]. We proceed as follows.\par

\indent A pair $<X, \tilde{X}>$ of Banach spaces will be called a \it dual 
pairing, \rm if
$$
(A)\hbox{  }\tilde{X}=X{'} \hbox{   or   }(B)\hbox{  }X=
\tilde{X}{'}.
$$

\indent In both cases, the canonical bilinear mapping is denoted by
$$
X\times\tilde{X}\to\mathbb{C} ,\hbox{   }(x,u)\to<x,u>.
$$

\indent If $<X,\tilde{X}>$ is a dual pairing, we consider the 
subalgebra ${\mathcal L}(X)$ of ${\rm L}(X)$ consisting of all 
operators $T\in{\rm L}(X)$ for which there is an operator $T{'}
\in {\rm L}(\tilde{X})$ with
$$
<Tx,u>=<x,T{'}u>,
$$
for all $x\in X$ and $u\in\tilde{X}$. It is clear that if the dual pairing 
is $<X,X{'}>$, then ${\mathcal L}(X)={\rm L}(X)$, and that if the 
dual pairing is $<X{'},X>$, then ${\mathcal L}(X)=\{T^*\colon T
\in {\rm L}(\tilde{X})\}$. In particular, each operator of the form
$$
f_{y,v}\colon X\to X,\hbox{}x\to<x,v>y,
$$
is contained in ${\mathcal L}(X)$, where $y\in X$ and $v\in\tilde{X}$.
\par 

\indent We now recall the definition of the tensor product given by J. 
Eschmeier in [5].\par

 Given two dual pairings $<X,\tilde{X}>$ and 
$<Y,\tilde{Y}>$, a tensor product of the Banach spaces $X$ and $Y$ 
relative to the dual pairings $<X,\tilde{X}>$ and 
$<Y,\tilde{Y}>$, is a Banach space $Z$ together with 
continuous bilinear mappings
\begin{align*}
&X\times Y\to Z,\hbox{}(x,y)\to x\otimes y;\\
&{\mathcal L}(X)
\times{\mathcal L}(Y)\to{\rm L}(Z),\hbox{}(T,S)\to T\otimes S,\\
\end{align*}
which satisfy the following conditions,

(T1)$\hbox{  }\parallel x\otimes y\parallel=\parallel x\parallel
\parallel y\parallel$,\par
(T2)$\hbox{  }T\otimes S(x\otimes y)=(Tx)\otimes(Sy)$,\par
(T3)$\hbox{  }(T_1\otimes S_1)\circ (T_2\otimes S_2)=(T_1T_2)
\otimes(S_1S_2),\hbox{}I\otimes I=I$,\par
(T4) $\hbox{\rm Im}(f_{x,u}\otimes I)\subseteq\{x\otimes y\colon y
\in Y\},\hbox{\rm Im}(f_{y,v}\otimes I)\subseteq\{x\otimes y\colon x
\in X\}$.\par

\indent In this work, as in [5], instead of $Z$ we shall often write 
$X\tilde{\otimes} Y$. In addition, as in [5], we shall have two 
applications of this definition of tensor product. 
First of all, the completion $X\tilde{\otimes_{\alpha}}Y$ of the 
algebraic tensor product of 
Banach spaces $X$ and $Y$ with respect to a quasi-uniform crossnorm 
$\alpha$ (see [6]) and an operator ideal between Banach spaces (see 
[5] and section 6).\par

\indent In section 4 and 5, given two complex Banach spaces $X_1$ and 
$X_2$, and two tuples of 
Banach spaces operators, $S$ and $T$, defined on $X_1$ and $X_2$ 
respectively, we shall describe the semi-Fredholm and the 
semi-Browder joint spectra of the tuple $(S\otimes I,I\otimes T)$,
whose operators, $S_i\otimes I$ and $I\otimes T_j$, $i=1,\ldots ,n$
and $j=1,\ldots ,m$, are defined on $X_1\tilde{\otimes }X_2$,
a tensor product of $X_1$ and $X_2$ relative to $<X_1,X_1{'}>$
and $<X_2,X_2{'}>$. However, in the following section,
we first introduce the split semi-Browder joint spectra, which will be necessary for our description.\par
  
\section{The split semi-Browder joint spectra}

\noindent In this section we introduce the upper and lower split semi-Browder joint spectra. We also prove
their main spectral properties. \par

\indent Let us consider, as in section 2, a complex Banach space $X$ and $T=(T_1,\ldots ,T_n)$ a commuting
tuple of operators defined on $X$. We say that
$T$ is \it lower split semi-Browder \rm if $R^{\infty}(T)$ has finite codimension and $N(\hat{T})$ has a direct
complement in $X^n$, where $\hat{T}\colon X^n\to X$ is the map considered in section 2. We denote by $\mathcal{SB}^{(n)}_{-}(X)$
the set of all lower split semi-Browder $n$-tuples, and the \it lower split semi-Browder spectrum \rm is the set
$$
sp_{{\mathcal B}_{-}}(T)=\{\lambda\in\mathbb{C}^n\colon T-\lambda\notin {\mathcal SB}^{(n)}_{-}(X)\}.
$$

\indent It is clear that
$$
sp_{{\mathcal B}_{-}}(T)=\sigma_{{\mathcal B}_{-}}(T)\cup {\mathcal C}_{-}(T),
$$
where ${\mathcal C}_{-}(T)=\{\lambda\in\mathbb{C}^n\colon N(\hat{T-\lambda})\hbox{ has not a direct complement in
 }X^n\}$. In particular, $sp_{{\mathcal B}_{-}}(T)$ is a nonempty set.\par

\indent On the other hand, if we consider the split defect spectrum and the essential split
defect spectrum of $T$ introduced in [4], $sp_{\delta}(T)$ and $sp_{\delta e}(T)$ respectively, 
sets that by [4, Theorem 2.7] may be presented as
$$
sp_{\delta}(T)=\sigma_{\delta}(T)\cup  {\mathcal C}_{-}(T),\hskip1cm sp_{\delta e}(T)=\sigma_{\Phi_{-}}(T)\cup  {\mathcal C}_{-}(T),
$$
then we have that
$$
sp_{\delta e}(T)\subseteq sp_{{\mathcal B}_{-}}(T)\subseteq sp_{\delta}(T).
$$
\indent In addition, if we consider the set $\tilde{{\mathcal A}}(T)=\{\lambda\in\mathbb{C}^n\colon \lambda\notin 
sp_{\delta  e}(T),  \forall k\in\mathbb{N}, 1\le \hbox{\rm codim} M_k(T-\lambda)<\infty ,\hbox{\rm codim} M_k(T-\lambda) \xrightarrow[k\to\infty]{} \infty \}$,
then it is clear that
$$
\tilde{{\mathcal A}}(T)\subseteq {\mathcal A}(T)\subseteq \sigma_{{\mathcal B}_{-}}(T)\subseteq sp_{{\mathcal B}_{-}}(T).
$$
\indent In particular 
$$
sp_{\delta e}(T)\cup \tilde{{\mathcal A}}(T)\subseteq  sp_{{\mathcal B}_{-}}(T).
$$

\indent On the other hand, let us consider $\lambda\in sp_{{\mathcal B}_{-}}(T)$, and let us decompose
the lower split semi-Browder spectrum of $T$ as
$$
sp_{{\mathcal B}_{-}}(T)=\sigma_{{\mathcal B}_{-}}(T)\cup{\mathcal C}_{-}(T)=\sigma_{{\Phi}_{-}}(T)\cup{\mathcal A}_{-}(T)\cup {\mathcal C}_{-}(T).
$$
Now, if $\lambda\in \sigma_{{\Phi}_{-}}(T)\cup{\mathcal C}_{-}(T)$, then $\lambda\in sp_{\delta e}(T)$. Moreover,
if $\lambda\in{\mathcal A}(T)\setminus(\sigma_{{\Phi}_{-}}(T)\cup{\mathcal C}_{-}(T))$, then $\lambda\in
{\mathcal A}(T)\setminus sp_{\delta e}(T)= \tilde{{\mathcal A}}(T)$. Thus, we have that
$$
sp_{{\mathcal B}_{-}}(T)= sp_{\delta e}(T)\cup \tilde{{\mathcal A}}(T).
$$

\indent We now introduce the upper split semi-Browder spectrum.\par

\indent If $X$ and $T=(T_1,\ldots ,T_n)$ are as above, then we say that
$T$ is \it upper split semi-Browder \rm  if it is upper semi-Browder and $R(\tilde{T})$ has a direct
complement in $X^n$, where $\tilde{T}\colon X\to X^n$ is the map considered in section 2. We denote by $\mathcal{SB}^{(n)}_{+}(X)$
the set of all upper split semi-Browder $n$-tuples, and the \it upper split semi-Browder spectrum \rm is the set
$$
sp_{{\mathcal B}_{+}}(T)=\{\lambda\in\mathbb{C}^n\colon T-\lambda\notin {\mathcal SB}^{(n)}_{+}(X)\}.
$$

\indent It is clear that
$$
sp_{{\mathcal B}_{+}}(T)=\sigma_{{\mathcal B}_{+}}(T)\cup {\mathcal C}_{+}(T),
$$
where ${\mathcal C}_{+}(T)=\{\lambda\in\mathbb{C}^n\colon R(\tilde{T-\lambda})\hbox{ has not a direct complement in
 }X^n\}$. In particular, $sp_{{\mathcal B}_{+}}(T)$ is a nonempty set.\par

\indent On the other hand, if we consider the split approximate point spectrum and the essential split
approximate point spectrum of $T$ (see [4]), $sp_{\pi}(T)$ and $sp_{\pi e}(T)$ respectively, i.e.,
the sets
$$
sp_{\pi}(T)=\sigma_{\pi}(T)\cup  {\mathcal C}_{+}(T),\hskip1cm sp_{\pi e}(T)=\sigma_{\Phi_{+}}(T)\cup  {\mathcal C}_{+}(T),
$$
then we have that
$$
sp_{\pi e}(T)\subseteq sp_{{\mathcal B}_{+}}(T)\subseteq sp_{\pi}(T).
$$
\indent In addition, if we consider the set $\tilde{{\mathcal D}}(T)=\{\lambda\in\mathbb{C}^n\colon \lambda\notin 
sp_{\pi  e}(T), \forall k\in\mathbb{N}, 1\le \dim N_k\tilde{(T-\lambda)}<\infty , \dim N_k\tilde{(T-\lambda)}\xrightarrow[k\to\infty]{ }\infty \}$,
then it is clear that
$$
\tilde{{\mathcal D}}(T)\subseteq {\mathcal D}(T)\subseteq \sigma_{{\mathcal B}_{+}}(T)\subseteq sp_{{\mathcal B}_{+}}(T).
$$
\indent In particular 
$$
sp_{\pi e}(T)\cup \tilde{{\mathcal D}}(T)\subseteq  sp_{{\mathcal B}_{+}}(T).
$$

\indent On the other hand, let us consider $\lambda\in sp_{{\mathcal B}_{+}}(T)$, and let us decompose
the upper split semi-Browder spectrum of $T$ as
$$
sp_{{\mathcal B}_{+}}(T)=\sigma_{{\mathcal B}_{+}}(T)\cup{\mathcal C}_{+}(T)=\sigma_{{\Phi}_{+}}(T)\cup{\mathcal D}_{+}(T)\cup{\mathcal C}_{+}(T).
$$
Now, if $\lambda\in \sigma_{{\Phi}_{+}}(T)\cup{\mathcal C}_{+}(T)$, 
then $\lambda\in sp_{\pi e}(T)$. Moreover,
if $\lambda\in{\mathcal D}(T)\setminus(\sigma_{{\Phi}_{+}}(T)\cup{\mathcal C}_{+}(T))$, 
then $\lambda\in
{\mathcal D}(T)\setminus sp_{\pi e}(T)= \tilde{{\mathcal D}}(T)$. Thus, we have that
$$
sp_{{\mathcal B}_{+}}(T)= sp_{\pi e}(T)\cup \tilde{{\mathcal D}}(T).
$$

\indent We now see that the sets that we have introduced satisfy the main spectral properties.\par

\begin{pro} Let $X$ be a complex Banach space and $T=(T_1,\ldots ,T_n)$ a commuting tuple
of bounded linear operators defined on $X$. Then the sets $sp_{{\mathcal B}_{-}}(T)$ and $sp_{{\mathcal B}_{+}}(T)$
are compact subsets of $\mathbb{C}^n$.\end{pro}

\begin{proof}

\indent Since $sp_{{\mathcal B}_{-}}(T)=sp_{\delta e}(T)\cup\tilde{{\mathcal A}}(T)\subseteq sp_{\delta e}(T)\cup
\sigma_{{\mathcal B}_{-}}(T)$, we have that $sp_{{\mathcal B}_{-}}(T)$ is a bounded subset of $\mathbb{C}^n$.\par

\indent On the other hand, let us consider a sequence $(\lambda_n)_{n\in\mathbb{N}}\subseteq sp_{{\mathcal B}_{-}}(T)$,
and $\lambda\in\mathbb{C}^n$ such that $\lambda_n\xrightarrow[n\to\infty]{ } \lambda$. 
If there exists a subsequence
$(\lambda_{n_k})_{k\in\mathbb{N}}\subseteq sp_{\delta e}(T)$, then $\lambda\in sp_{\delta e}(T)\subseteq 
sp_{{\mathcal B}_{-}}(T)$. Thus, we may suppose that there is $n_0\in\mathbb{N}$ such that for all $n\in \mathbb{N}$,
$n\ge n_0$, $\lambda_n\in \tilde{{\mathcal A}}(T)$. Moreover, we may also suppose that $\lambda\notin sp_{\delta e}(T)$.
In particular, there is an open neighborhood of $\lambda$, $U$, such 
that $U\cap sp_{\delta e}(T)=\emptyset
$,
and there is $n_1\in\mathbb{N}$ such that $\lambda_n\in U$, for all $n\ge n_1$.\par

\indent However, since for all $n\ge n_0$, $\lambda_n\in\tilde{{\mathcal A}}(T)\subseteq {\mathcal A}(T)\subseteq 
\sigma_{{\mathcal B}_{-}}(T)$, then $\lambda\in \sigma_{{\mathcal B}_{-}}(T)$. But $\lambda\notin\sigma_{\Phi_{-}}(T)$,
for $\sigma_{\Phi_{-}}(T)\subseteq sp_{\delta e}(T)$. Then, $\lambda\in {\mathcal A}(T)\setminus sp_{\delta e}(T)=
\tilde{{\mathcal A}}(T)\subseteq sp_{{\mathcal B}_{-}}(T)$.\par

\indent By means of a similar argument, it is possible to see that $sp_{{\mathcal B}_{+}}(T)$ is a compact subset of
$\mathbb{C}^n$.
\end{proof}  

\begin{pro} Let $X$ be a complex Banach space and 
$T=(T_1,\ldots ,T_n, T_{n+1})$ a commuting tuple
of bounded linear operators defined on $X$. 
If $\pi\colon\mathbb{C}^{n+1}\to \mathbb{C}^n$ denotes the projection
onto the first $n$-coordinate, then we have that \par
{\rm (i)} $\pi (sp_{{\mathcal B}_{-}}(T_1,\ldots ,T_n,T_{n+1}))=
sp_{{\mathcal B}_{-}}(T_1,\ldots ,T_n),$ \par
{\rm (ii)} $\pi (sp_{{\mathcal B}_{+}}(T_1,\ldots ,T_n,T_{n+1}))=
sp_{{\mathcal B}_{+}}(T_1,\ldots ,T_n).$\par
\end{pro}
    
\begin{proof}

\indent By [7, Corollary 7] we now that $\pi (\sigma_{{\mathcal B}_{-}}(T_1,\ldots ,
T_n, T_{n+1}))=\sigma_{{\mathcal B}_{-}}(T_1,\ldots ,T_n)$
$\subseteq sp_{{\mathcal B}_{-}}(T_1,\ldots ,T_n)$. Moreover, since 
${\mathcal C}_{-}(T_1,\ldots ,T_n,T_{n+1})\subseteq 
sp_{\delta e}(T_1,\ldots ,T_n,T_{n+1})$, by [4, Corollary 2.6] we have that
\begin{align*}
\pi ({\mathcal C}_{-}(T_1,\ldots ,T_n,T_{n+1}))&\subseteq 
\pi (sp_{\delta e}(T_1,\ldots ,T_n,T_{n+1}))=
sp_{\delta e}(T_1,\ldots ,T_n)\\
&\subseteq sp_{{\mathcal B}_{-}}(T_1,\ldots ,T_n).
\end{align*}
\indent Thus, we have that
$$
\pi (sp_{{\mathcal B}_{-}}(T_1,\ldots ,T_n,T_{n+1}))\subseteq
sp_{{\mathcal B}_{-}}(T_1,\ldots ,T_n).
$$

\indent On the other hand, by [7, Corollary 7] we also have that
$$
\sigma_{{\mathcal B}_{-}}(T_1,\ldots ,T_n)=\pi (\sigma_{{\mathcal B}_{-}}
(T_1,\ldots ,T_n,T_{n+1}))\subseteq\pi (sp_{{\mathcal B}_{-}}
(T_1,\ldots ,T_n,T_{n+1})).
$$

Furthermore, since ${\mathcal C}_{-}
(T_1,\ldots ,T_n)\subseteq 
sp_{\delta e}(T_1,\ldots ,T_n)$,
by [4, Corollary 2.6] we also have that
${\mathcal C}_{-}(T_1,\ldots ,T_n)\subseteq sp_{\delta e}(T_1,\ldots ,T_n)=
\pi (sp_{\delta e}(T_1,\ldots ,T_n, T_{n+1}))\subseteq 
\pi(sp_{{\mathcal B}_{-}}(T_1,\ldots ,T_n, T_{n+1}))$.
Thus, 
$$
sp_{{\mathcal B}_{-}}(T_1,\ldots ,T_n)\subseteq
\pi (sp_{{\mathcal B}_{-}}(T_1,\ldots ,T_n,T_{n+1})),
$$
i.e., we have proved the first statement of the proposition.\par
 
\indent By means of a similar argument it is possible to see the second statement.
\end{proof}
    
\indent In the following proposition we shall see that the split 
semi-Browder joint spectra satisfy the analytic spectral mapping theorem.\par

\begin{pro}Let $X$ be a complex Banach space and 
$T=(T_1,\ldots ,T_n)$ a commuting tuple
of bounded linear operators defined on $X$. Then, if $f\in 
{\mathcal O}(sp(T))^m$, we have that \par
\item {\rm (i)} $ f(sp_{{\mathcal B}_{-}}(T_1,\ldots ,T_n))=sp_{{\mathcal B}_{-}}(f(T_1,\ldots ,T_n)),
$\par
\item {\rm (ii)} $f(sp_{{\mathcal B}_{+}}(T_1,\ldots ,T_n))=sp_{{\mathcal B}_{+}}(f(T_1,\ldots ,T_n)),$\par
\noindent where $sp(T)$ denotes the split spectrum of $T$.
\end{pro}
\begin{proof}

\indent By [4, Corollary 2.6], the split sectrum of $T$, $sp(T)$, satisfies the 
analytic spectral mapping theorem, i.e., there is an algebra morphism
$$
\Phi\colon {\mathcal O}(sp(T))\to L(X),\hskip1cm f\to f(T),
$$
such that $1(T)=I$, $z_i(T)=T_i$, $1\le i\le n$,
where $z_i$ denotes the projection of $\mathbb{C}^n$ onto the
$i$-th coordinate,
and such that the equality $sp(f(T))=f(sp(T))$, holds for all 
$f\in {\mathcal O}(sp(T))^m$.\par

\indent Now, as in [4], let us consider the algebra 
$$
A=\overline{\Phi ({\mathcal O}(sp(T)))}\subseteq L(X).
$$ 
Then, we have that the split spectrum is a spectral system on $A$, 
in the sense of [4, section 1].\par

\indent In order to show this claim, since the split spectrum is a compact set which
also satisfies the projection property ([4, Corollary 2.6]), we have only to see that if 
$a=(a_1,\ldots ,a_n)$ is a tuple of commuting operators such that 
$a_i\in A$, then $sp(a)\subseteq\sigma^A_{joint (a)}$
(the usual joint spectrum of $a$ with respect to the algebra $A$, see  [4, section 1]). \par

\indent In fact, if $\lambda=(\lambda_1,\ldots ,\lambda_n) \in sp(a)
\setminus \sigma^A_{joint (a)}$, then there are $B_1,\ldots ,B_n\in A$
such that $\sum_{i=1}^n B_i(a_i-\lambda_iI)=I$, where $I$ denotes
the identity map of $X$. In particular, 
$$
\sum_{i=1}^n L_{B_i}(L_{a_i}-\lambda_i I_{L(X)})=I_{L(X)}
$$.
Then, $\lambda\notin\sigma (L_a)$, the Taylor joint spectrum of the
tuple of left multiplication, $L_a=(L_{a_1},\ldots ,L_{a_n})$, defined on 
$L(X)$. However, by [4, Corollary 2.5], $\lambda\notin sp(a)$, which is
impossible by our assumption.\par 

\indent Now, since $sp_{{\mathcal B}_{-}}(T_1,\ldots ,T_n)$
and $sp_{{\mathcal B}_{+}}(T_1,\ldots ,T_n)$ are contained in $sp (T)$,
by Propositions 3.1 and 3.2, $sp_{{\mathcal B}_{-}}(T_1,\ldots ,T_n)$
and $sp_{{\mathcal B}_{+}}(T_1,\ldots ,T_n)$ are spectral systems on $A$ 
contained in $sp(T)$. Then, by [4, Theorem 1.2] and [4, Corollary 1.3], since the split spectrum is
a spectral system on $A$ which satisfy the analytic spectral mapping theorem,
$sp_{{\mathcal B}_{-}}(T_1,\ldots ,T_n)$
and $sp_{{\mathcal B}_{+}}(T_1,\ldots ,T_n)$ also satisfy the analytic
spectral mapping theorem defined on ${\mathcal O}(sp(T))$.
\end{proof}

\indent In the following section we give a description of the semi-Fredholm
joint spectra of the system $(S\otimes I,I\otimes T)$, which will be a central
step for one of the main theorems of the present article.\par

\section{ The semi-Fredholm joint spectra}

\noindent In this section we consider two complex Banach spaces
$X_1$ and $X_2$, two tuples of bounded linear operators
defined on $X_1$ and $X_2$, $S=(S_1, \ldots ,S_n)$
and $T=(T_1,\ldots ,T_n)$ respectively, and we describe the semi-Fredholm
joint spectra of the $(n+m)$-tuple $(S\otimes I,I\otimes T)$
defined on $X_1\tilde{\otimes }X_2$, a tensor product between
$X_1$ and $X_2$ relative to $<X_1,X_1{'}>$ and $<X_2,X_2{'}>$,
where $(S\otimes I,I\otimes T)=
(S_1\otimes I,\ldots ,S_n\otimes I,I\otimes T_1,\ldots ,I\otimes T_m)$.
\par

\indent We recall that if $K_1$, $K_2$ and $K$ are the Koszul complexes
associated to the tuples $S$, $T$ and $(S\otimes I,I\otimes T)$ 
respectively (see [9]), i.e., 
$K_1=(X_1\otimes \wedge \mathbb{C}^n,d_1)$,  
$K_2=(X_2\otimes \wedge \mathbb{C}^m,d_2)$ and 
$K=(X_1\tilde{\otimes}X_2\otimes \wedge \mathbb{C}^{n+m},d_{12})$,
then, by [5, section 3] we have that $K$ is isomorphic to the total 
complex of the double complex obtained from the tensor
product of the complexes $K_1$ and $K_2$; we denote
this total complex by $K_1\tilde{\otimes}K_2$. Moreover,
if we consider the differential spaces associated to $K_1$, $K_2$,
$K$, and $K_1\tilde{\otimes}K_2$, which we denote, by
${\mathcal K}_1$, ${\mathcal K}_2$, ${\mathcal K}$, and ${\mathcal K}_1
\tilde{\otimes} {\mathcal K}_2$ respectively, then we have that
${\mathcal K}\cong {\mathcal K_1}\tilde{\otimes}{\mathcal K_2}$, 
and if the boundary of these differential
spaces are, $\partial_1$, $\partial_2$, $\partial_{12}$, and
$\partial$ respectively, then we have that $\partial = \partial_1\otimes I+\eta
\otimes \partial_2$, where $\eta$ is the map, $\eta\colon {\mathcal K}_2\to {\mathcal K}_2$,
$\eta\mid X_2\otimes\wedge^m \mathbb{C}=(-1)^mI$ (for a 
complete exposition see [5, section 3]).\par     

\indent In the following proposition we describe  the defect,
the approximate point spectrum, and the split version of these
spectra for the tuple $(S\otimes I,I\otimes T)$. This result is 
necessary for our description of the semi-Fredholm joint spectra.\par

\begin{pro} Let $X_1$ and $X_2$ be two complex Banach
spaces, and $X_1\tilde{\otimes}X_2$ a tensor product of $X_1$
and $X_2$ relative to $<X_1,X_1{'}>$ and $<X_2,X_2{'}>$. Let us
consider two tuples of commuting operators defined on $X_1$ and 
$X_2$, $S$ and $T$ respectively. Then, for the tuple $(S\otimes I,
I\otimes T)$, defined on $X_1\tilde{\otimes}X_2$, we have that \par
{\rm (i)} $\sigma_{\delta}(S)\times \sigma_{\delta }(T)\subseteq
\sigma_{\delta}(S\otimes I,I\otimes T)\subseteq sp_{\delta}
(S\otimes I,I\otimes T)
 \subseteq sp_{\delta}(S)
\times sp_{\delta}(T)$,\par
{\rm (ii)}$\sigma_{\pi}(S)\times \sigma_{\pi }(T)\subseteq
\sigma_{\pi}(S\otimes I,I\otimes T) \subseteq sp_{\pi }
(S\otimes I,I\otimes T)\subseteq sp_{\pi}(S)\times sp_{\pi}(T)$.\par
\noindent In addition, if $X_1$ and $X_2$ are Hilbert spaces, the above inclusions
are equalities.
\end{pro}

\begin{proof}

\indent Let us consider $\lambda\in\mathbb{C}^n$, $\mu\in\mathbb{C}^m$
and the Koszul complexes associated to $S-\lambda $, $T-\mu $
and $(S\otimes I,I\otimes T) -(\lambda ,\mu)=((S-\lambda)\otimes I ,
I\otimes (T-\mu))$, which we
denote by $K_1$, $K_2$ and $K$.
By the previous  observation we have that $K\cong 
K_1\tilde{\otimes}K_2$. Moreover, if we consider the
differential spaces associated to these complexes, ${\mathcal K}_1$
${\mathcal K}_2$, and ${\mathcal K}$, then we have that ${\mathcal K}\cong
{\mathcal K}_1\tilde{\otimes}{\mathcal K}_2$. \par

\indent Now, we may apply [5, Theorem 2.2] to the differential spaces
${\mathcal K}_1$, ${\mathcal K}_2$, and 
${\mathcal K}_1\tilde{\otimes}{\mathcal K}_2$. However, by the definition
of the map $\varphi$ in [5, Theorem 2.2], the grading of the differential spaces ${\mathcal K}_1$, ${\mathcal K}_2$, and 
${\mathcal K}_1\tilde{\otimes}{\mathcal K}_2$, and of the isomorphism 
${\mathcal K}\cong {\mathcal K}_1\tilde{\otimes}{\mathcal K}_2$, we have
the left hand side inclusion of the first statement.\par

\indent The middle inclusion is clear.\par

\indent Let us now suppose that $(\lambda ,\mu)\notin 
sp_{\delta}(S)\times sp_{\delta}(T)$. Then, either $\lambda\notin
sp_{\delta}(S)$ or $\mu\notin sp_{\delta }(T)$. We shall see that
if $\lambda\notin sp_{\delta}(S)$, then $(\lambda ,\mu)\notin 
sp_{\delta}(S\otimes I, I\otimes T)$. By means of a similar
argumet it is possible to see that if $\mu\notin sp_{\delta}(T)$ 
then $(\lambda ,\mu)\notin sp_{\delta}(S\otimes I,I\otimes T)$.\par

\indent Now, if $\lambda\notin sp_{\delta}(S)$, there is a bounded
linear operator $h\colon X_1\to X_1\otimes\wedge^n\mathbb{C}$ such that
$$
d_{11}\circ h=I,
$$
where $d_{11}\colon X_1\otimes \wedge \mathbb{C}^n\to X_1$ is the chain map of
the Koszul complex $K_1$ at level $p=1$.\par

\indent Let us consider the map 
$$
H\colon X_1\tilde{\otimes} X_2\to X_1\otimes\wedge\mathbb{C}^n
\tilde{\otimes}X_2 , \hskip.5cm H=h\otimes I.
$$
Then, by the properties of the tensor product introduced in [5],
$H$ is a well defined map which satisfies
$$
d_1\circ H=d_{11}\circ h\otimes I=
I\otimes I=I,
$$
where $ d_1$ is the chain map of the complex 
$K_1\tilde{\otimes} K_2$ at level $p=1$. Since $K
\cong K_1\tilde{\otimes} K_2$, we have that
$(\lambda ,\mu)\notin sp_{\delta }(S\otimes I,I\otimes T) $.\par

\indent The second statement may be proved by means of a similar argument.
\end{proof}

\indent In the following proposition we state our description of the
semi-Fredholm joint spectra.\par

\begin{pro} Let $X_1$ and $X_2$ be two complex Banach
spaces, and $X_1\tilde{\otimes}X_2$ a tensor product of $X_1$
and $X_2$ relative to $<X_1,X_1{'}>$ and $<X_2,X_2{'}>$. Let us
consider two tuples of commuting operators defined on $X_1$ and 
$X_2$, $S$ and $T$ respectively. Then, for the tuple $(S\otimes I,
I\otimes T)$, defined on $X_1\tilde{\otimes}X_2$, we have that

\begin{align*}
{\rm (i)}\hbox{  } &\sigma_{\Phi_{-}}(S)\times \sigma_{\delta }(T)\cup 
\sigma_{\delta}(S)\times\sigma_{\Phi_{-}}(T)\subseteq
\sigma_{\Phi_{-}}(S\otimes I,I\otimes T)\subseteq\\
& sp_{\delta e}(S\otimes I,I\otimes T)
\subseteq sp_{\delta e }(S)\times sp_{\delta}(T)\cup
sp_{\delta  }(S)\times sp_{\delta e}(T),\\\end{align*}

\begin{align*}
{\rm (ii)}\hbox{  }&\sigma_{\Phi_{+}}(S)\times \sigma_{\pi }(T)\cup 
\sigma_{\pi}(S)\times\sigma_{\Phi_{+}}(T)\subseteq
\sigma_{\Phi_{+}}(S\otimes I,I\otimes T)\subseteq\\
& sp_{\pi e}(S\otimes I,I\otimes T)
\subseteq sp_{\pi e }(S)\times sp_{\pi}(T)\cup
sp_{\pi  }(S)\times sp_{\pi e}(T).\\\end{align*}

\noindent In addition, if $X_1$ and $X_2$ are Hilbert spaces,
the above inclusions are equalities.
\end{pro}

\begin{proof}
\indent First of all, let us observe that we use the same notations of Proposition 4.1.\par

\indent  With regard to the first statement, in order to prove the 
left hand side inclusion, it is enough to adapt for this case
the argument that we have developed in Proposition 4.1 for the corresponding inclusion.\par

\indent The middle inclusion is clear.\par

\indent Let us denote by $E$ the set $E= sp_{\delta e }(S)\times 
sp_{\delta}(T)\cup sp_{\delta  }(S)\times sp_{\delta e}(T)$,
and let us consider $(\lambda ,\mu)\in sp_{\delta e}(S\otimes I,I\otimes T)
\setminus E$. Then, by Proposition 4.1, since $(\lambda ,\mu)\in sp_{\delta }
(S\otimes I,I\otimes T)\subseteq sp_{\delta }(S)\times 
sp_{\delta}(T)$, we have that $\lambda\in sp_{\delta}(S)\setminus
sp_{\delta e}(S)$ and $\mu\in sp_{\delta}(T)\setminus
sp_{\delta e}(T)$. In particular, there are two linear bounded maps
$h\colon X_1\to X_1\otimes\wedge^1\mathbb{C}^n$, $g\colon X_2\to 
X_2\otimes\wedge^1\mathbb{C}^m$, and two compact operators
$k_1\colon X_1\to X_1$ and $k_2\colon X_2\to X_2$ such that
$$
d_{11}\circ h=I-k_1,\hskip.5cm d_{21}\circ g=
I-k_2,
$$  
where $d_{21}$ is the boundary map of the complex $K_2$ at level $p=1$.
Moreover, by an argument similar to [4, Theorem 2.7] or [5, Proposition 2.1], the maps
$k_i$, $i=1$, $2$, may be chosen as finite rank projectors.\par

\indent In addition, by the properties of the tensor product introduced
in [5], we may
consider the well defined map 
$$
H\colon X_1\tilde{\otimes} X_2 \to (K_1\tilde{\otimes}
K_2)_1,\hskip.5cm H=(h\otimes I,I\otimes g).
$$
Now, an easy calculation shows that $d_1\circ H
=I-k_1\otimes k_2$, where $d_1$ denotes the chain map of
the complex $K_1\tilde{\otimes}K_2$ at level $p=1$. 
However, it is not difficult to see, using in 
particular [5, Lemma 1.1], that $k_1\otimes k_2$ is a finite rank projector
whose range coincide with $R(k_1)\otimes R(k_2)$. In particular,
$k_1\otimes k_2$ is a compact operator. 
Thus, since $K\cong K_1\tilde{\otimes} K_2$, $(\lambda ,\mu)\notin sp_{\delta e}(S\otimes I,I\otimes T )$,
which is impossible by our assumptions.\par

\indent By means of a similar argument it is possible to see the second
statement.
\end{proof}
 
\section{The semi-Browder joint spectra}

\noindent In this section we give our description of the semi-Browder
joint spectra of the tuple $(S\otimes I  ,I\otimes T)$. The following
theorem is one of the main results of the present article.\par

\begin{thm}Let $X_1$ and $X_2$ be two complex Banach
spaces, and $X_1\tilde{\otimes}X_2$ a tensor product of $X_1$
and $X_2$ relative to $<X_1,X_1{'}>$ and $<X_2,X_2{'}>$. Let us
consider two tuples of commuting operators defined on $X_1$ and 
$X_2$, $S$ and $T$ respectively. Then, for the tuple $(S\otimes I,
I\otimes T)$, defined on $X_1\tilde{\otimes}X_2$, we have that
\begin{align*}
{\rm (i)}\hbox{  }&\sigma_{{\mathcal B}_{-}}(S)\times \sigma_{\delta }(T)\cup 
\sigma_{\delta}(S)\times\sigma_{{\mathcal B}_{-}}(T)\subseteq
\sigma_{{\mathcal B}_{-}}(S\otimes I,I\otimes T)\subseteq\\
& sp_{{\mathcal B}_{-}}(S\otimes I,I\otimes T)
\subseteq sp_{{\mathcal B}_{-} }(S)\times sp_{\delta}(T)\cup
sp_{\delta  }(S)\times sp_{{\mathcal B}_{-}}(T),\\\end{align*}

\begin{align*}
{\rm (ii)}\hbox{  }&\sigma_{{\mathcal B}_{+}}(S)\times \sigma_{\pi }(T)\cup 
\sigma_{\pi}(S)\times\sigma_{{\mathcal B}_{+}}(T)\subseteq
\sigma_{{\mathcal B}_{+}}(S\otimes I,I\otimes T)\subseteq\\
& sp_{{\mathcal B}_{+}}(S\otimes I,I\otimes T)
\subseteq sp_{{\mathcal B}_{+} }(S)\times sp_{\pi}(T)\cup
sp_{\pi  }(S)\times sp_{{\mathcal B}_{+}}(T),\\\end{align*}

\noindent In addition, if $X_1$ and $X_2$ are Hilbert spaces,
the above inclusions are equalities.
\end{thm}
\begin{proof} 
\indent First of all, as in Proposition 4.2, we use the notations of Proposition 4.1.\par

\indent Let us consider $(\lambda ,\mu)\in 
\sigma_{{\mathcal B}_{-}}(S)\times \sigma_{\delta }(T)$.
If $\lambda\in \sigma_{\Phi_{-}}(S)$, then, by Proposition 4.2, $(\lambda ,\mu)
\in\sigma_{\Phi_{-}}(S)\times \sigma_{\delta }(T)\subseteq
\sigma_{\Phi_{-}}(S\otimes I,I\otimes T)\subseteq
\sigma_{{\mathcal B}_{-}}(S\otimes I,I\otimes T)$.\par

\indent Now, if $\lambda\in{\mathcal A }(S)$, since $\mu\in 
\sigma_{\delta}(T)$, by the definition of the map $\varphi$ in [5, Theorem 2.2], the 
grading of the complex $K_1$, $K_2$, and $K_1\tilde{\otimes}K_2$, and by
the isomorphism $K\cong K_1\tilde{\otimes}K_2$, we have that
$\dim H_0(K)=\dim H_0(K_1\tilde{\otimes} K_2)\ge \dim H_0
(K_1)\times \dim H_0(K_2)\ge 1$. In particular,
$(\lambda ,\mu)\in\sigma_{\delta}(S\otimes I ,I\otimes T)$.\par

\indent Moreover, if $\dim H_0(K)=\infty$, then 
$(\lambda,\mu)\in\sigma_{\Phi_{-}}(S\otimes I,I\otimes T)
\subseteq\sigma_{{\mathcal B}_{-}}(S\otimes I,I\otimes T)$.\par

\indent On the other hand, if we suppose that $(\lambda ,\mu)\notin 
\sigma_{\Phi_{-}}(S\otimes I,I\otimes T)$. Then, we consider 
the tuples of operators 
$(S-\lambda)^l=((S_1-\lambda_1)^l ,\ldots ,(S_n-\lambda_n )^l)$
and $(T-\mu)^l=((T_1-\mu_1)^l ,\ldots ,(T_m-\mu_m )^l)$,
and we denote by $K_1^l$ and $K_2^l$ the 
Koszul complexes associated to the tuples $(S-\lambda)^l$ and
$(T-\mu)^l$, respectively. Moreover, if we denote by $K^l$ the Koszul complex 
associated to the tuple $((S-\lambda)^l\otimes I, 
I\otimes (T-\mu)^l)$, as above, $K^l$ is isomorphic to the total complex of the double complex 
of the tensor product of $K_1^l$ and $K_2^l$, i.e.,
$K^l\cong K_1^l\tilde{\otimes}K_2^l$.\par

\indent In addition, as we have seen 
for the complexes $K_1$, $K_2$, $K$, and $K_1\tilde{\otimes}K_2$, we have that $\dim H_0(K^l)=
\dim H_0(K_1^l\tilde{\otimes} K_2^l)
\ge \dim H_0(K_1^l)\times \dim H_0(K_2^l)$.
Now, since $\mu\in\sigma_{\delta}(T)$, by the analytic spectral
mapping theorem for the defect spectrum (see [4, Corollary 2.1], we have that
$\dim H_0(K_2^l)\neq 0$. In addition, since $\dim H_0(K_1^l)=
\hbox{\rm codim} M_l(S-\lambda)$, and since
$\lambda\in{\mathcal A}(S)$, then $\dim H_0(K^l)\xrightarrow[l\to \infty]{ }
\infty$. However, $\dim H_0(K^l)=\hbox{\rm codim } 
M_l((S-\lambda)\otimes I, I\otimes (T-\mu))$. In particular, $(\lambda ,\mu)\in
{\mathcal A}(S\otimes I ,I\otimes T)\subseteq \sigma_{{\mathcal B}_{-}}
(S\otimes I, I\otimes T)$.\par

\indent By means of a similar argument it is possible to see that
$\sigma_{\delta}(S)\times\sigma_{{\mathcal B}_{-}}(T)\subseteq
\sigma_{{\mathcal B}_{-}}(S\otimes I,I\otimes T)$.\par

\indent The middle inclusion is clear.\par

\indent  In order to see the right hand inclusion, let us consider $(\lambda ,\mu)\in 
sp_{{\mathcal B}_{-}}(S\otimes I,I\otimes T)$. If $(\lambda ,\mu)\in
sp_{\delta e}(S\otimes I,I\otimes T)$, then by Proposition 4.2,
$(\lambda ,\mu)\in sp_{\delta e }(S)\times sp_{\delta}(T)\cup
sp_{\delta  }(S)\times sp_{\delta e}(T)\subseteq 
sp_{{\mathcal B}_{-} }(S)\times sp_{\delta}(T)\cup
sp_{\delta  }(S)\times sp_{{\mathcal B}_{-}}(T)$.\par

\indent On the other hand, if $(\lambda ,\mu)\in \tilde{\mathcal A}
(S\otimes I,I\otimes T)$, since by Proposition 4.1 $sp_{{\mathcal B}_{-}}
(S\otimes I,I\otimes T)\subseteq sp_{\delta}(S\otimes I,I\otimes T)
\subseteq sp_{\delta}(S)\times sp_{\delta }(T)$,
if $(\lambda ,\mu)\notin (sp_{{\mathcal B}_{-} }(S)\times sp_{\delta}(T)
\cup sp_{\delta  }(S)\times sp_{{\mathcal B}_{-}}(T))$, then
$\lambda\notin sp_{{\mathcal B}_{-} }(S)$ and $\mu\notin 
sp_{{\mathcal B}_{-} }(T)$. In particular, $\lambda\notin 
sp_{\delta e}(S)$ and $\mu\notin sp_{\delta e}(T)$, 
and there is $l\in\mathbb{N}$ such that for all $r\ge l$,
$\dim H_0(K_1^r)=\dim H_0(K_1^l)$
and $\dim H_0(K_2^r)=\dim H_0(K_2^l)$.\par

\indent In addition, by the analytic spectral mapping theorem of 
the essential split defect spectrum ([4, Corollary 2.6]), the complex $K_1^r$ and 
$K_2^r$ are Fredholm split for all $r\in\mathbb{N}$ at level $p=0$. 
In particular, for all $r\in\mathbb{N}$ there are bounded linear maps 
$h_r\colon X_1\to X_1\otimes\wedge^1 \mathbb{C}^n$
and $g_r\colon X_2\to X_2\otimes\wedge^1 \mathbb{C}^m$,
and finite rank projectors (see Proposition 4.2),
$k_{1r}\colon X_1\to X_1$ and
$k_{2r}\colon X_2\to X_2$, such that
$$
d_{11}^r\circ h_r=I-k_{1r},\hskip.5cm 
d_{21}^r\circ g_r=I-k_{2r},
$$
where $d_{11}^r$ and $d_{21}^r$ are the
chain maps of the complex $K_1^r$ and $K_2^r$
at level $p=1$, respectively.\par

\indent  Moreover, since the complexes $K_1^r$ and $K_2^r$ are Fredholm split at level $p=0$, by [4, Theorem 2.7]
the complexes $K_1^r$ and $K_2^r$ are Fredholm at level $p=0$ and 
$N(d_{11}^r)$ and $N(d_{21}^r)$ have
direct complements in $X_1\otimes \wedge^1\mathbb{C}^n$ and $X_2\otimes \wedge^1\mathbb{C}^m$ respectively.
Now, by an argument similar to [4, Theorem 2.7] or [5, Proposition 2.1], we have that
the maps $h_r$, $g_r$ $k_{1r}$ and $k_{2r}$ may be 
chosen in the following way. If $N_1^r$ and $N_2^r$ are finite dimensional
subspaces of $X_1$ and $X_2$ respectively, such that $R(d_{11}^r
)\oplus N_1^r=X_1$ and $R(d_{21}^r)
\oplus N_2^r=X_2$ and $L_1^r$ and $L_2^r$ are closed linear subspaces of
$X_1\otimes\wedge^1\mathbb{C}^n$ and  
$X_2\otimes\wedge^1\mathbb{C}^m$ respectively, such that
$N (d_{11}^r)\oplus L_1^r= X_1\otimes
\wedge^1\mathbb{C}^n$, and $N (d_{21}^r)\oplus L_2^r=
X_2\otimes
\wedge^1\mathbb{C}^m$, then, $k_{1r}$, respectively $k_{2r}$, may be chosen
as the projector onto $N_1^r$, respectively $N_2^r$, whose null space
coincide with $R(d_{11}^r)$, respectively  $R(d_{21}^r)$,
and the map $h_r$, respectively $g_r$, may be chosen such that 
$h_r\circ d_{11}^r=I\mid L_1^r$, respectively
$g_r\circ d_{21}^r=I\mid L_2^r$, $h_r\mid N_1^r=0$, respectively
$g_r\mid N_2^r=0$. In particular, $R(k_{1r})\cong
H_0(K_1^r)$ and $R(k_{2r})\cong
H_0(K_2^r)$.\par        

\indent Now, as in Proposition 4.2, for all
$r\in\mathbb{N}$ we have a well defined map $H_r\colon X_1\tilde{\otimes}
X_2\to (K_1^r\tilde{\otimes}K_2^r)_1$ such that
$$
d^r_1\circ H_r =I-k_{1r}\otimes k_{2r},
$$
where $d^r_1$ is the boundary map of the complex $K_1^r\tilde{\otimes}K_2^r$
at level $p=1$.\par  
\indent Then, since for all $r\in \mathbb{N}$, $R(k_{1r}\otimes k_{2r}) =
R(k_{1r})\otimes R(k_{2r})$ (see Proposition 4.2), for all $r\ge l$ we have that

\begin{align*} 
\dim H_0(K^r)&=\dim H_0(K_1^r\tilde
{\otimes}K_2^r)\le \dim R(k_{1r}\otimes k_{2r})\\
& =\dim R(k_{1r})\times \dim R(k_{2r})=\dim H_0(K_1^r)
\times \dim H_0(K_2^r)\\
&=\dim H_0(K_1^l)
\times \dim H_0(K_2^l),\\\end{align*}

\noindent which is impossible for $(\lambda ,\mu)\in \tilde{{\mathcal A}}(S\otimes I,
I\otimes T)$ and $\dim H_0(K^r)=\hbox{\rm codim }M_r((S-\lambda)\otimes I ,I\otimes (T-\mu))$.\par

\indent By means of a similar argument it is possible to see
the second statement.
\end{proof}
 
\section{Operator ideals between Banach spaces}

\noindent In this section we extend our descriptions of the semi-Fredholm 
joint spectra and the semi-Browder joint spectra for  
tuples of left and right multiplications defined on an operator ideal between Banach spaces in the sense
of [5]. We first recall the definition of such an ideal and then we
introduce the tuples
with which we shall work. For a complete exposition see [5].\par

\indent An operator ideal $J$ between Banach spaces $X_2$ 
and $X_1$ will be a \it linear subspace \rm of ${\rm L}(X_2,X_1)$, equiped 
with a 
space norm $\alpha$ such that\par
(i)  $x_1\otimes x_2{'}\in J$ and $\alpha (x_1\otimes x_2{'})=
\parallel x_1
\parallel \parallel x_2\parallel$,\par
(ii) $SAT\in J$ and $\alpha (SAT)\le \parallel S\parallel \alpha (A) 
\parallel T\parallel$, \par

\noindent where $x_1\in X_1$, $x_2{'}\in X_2{'}$, $A\in J$, 
$S\in {\rm L}
(X_1)$,
$T\in{\rm L}(X_2)$, and $x_1\otimes x_2{'}$ is the usual rank one
operator $X_2\to X_1$, $x_2\to <x_2,x_2{'}>x_1$.\par

\indent Examples of this kind of ideals are given in [5, section 1].\par

\indent Let us recall that such operator ideal $J$ is naturally a tensor
product relative to $<X_1,X_1{'}>$ and $<X_2{'},X_2>$, with the 
bilinear mappings
$$
X_1\times X_2{'}\to J,\hbox{  } (x_1,x_2{'})\to x_1\otimes x_2{'},
$$
$$
{\mathcal L}(X_1)\times{\mathcal L}(X_2{'})\to {\rm L}(J),
\hbox{  } (S,T{'})\to S\otimes T{'},
$$
where $S\otimes T{'} (A)=SAT$.\par 

\indent On the other hand, if $X$ is a Banach space and $U\in{\rm L}
(X)$, we denote by $L_U$ and $R_U$ the operators of left and 
right multiplication in ${\rm L}(X)$, respectively, i.e., if $V\in {\rm L}(X)$, then
$L_U(V)= UV$ and $R_U(V)=VU$. \par  

\indent Now, if $S=(S_1,\ldots ,S_n)$ and $T=(T_1,\ldots , 
T_m)$ are tuples of commuting operators defined on $X_1$ and
$X_2$ respectively, if $J$ is seen as a tensor product of $X_1$ and $X_2$
relative to $<X_1,X_1{'}>$ and $<X_2{'},X_2>$, then the tuple of left 
and right multilications $(L_S,R_T)$ defined on $L(J)$, 
$(L_S, R_T)=(L_{S_1},\ldots ,L_{S_n}, R_{T_1},\ldots ,R_{T_m})$,
may be identified with the $(n+m)$-tuple $(S\otimes I, I\otimes T{'})$
defined on $X_1\tilde{\otimes}X_2{'}$, where $T{'}=(T_1{'},\ldots ,
T_m{'})$ and for all $i=1,\ldots ,m$, $T_i{'}$ is the adjoint map
associated to $T_i$ (see [5, Theorem 3.1]).\par

\indent In addition, if $\lambda\in \mathbb{C}^n$ and $\mu\in\mathbb{C}^m$,
and if we denote by $K_1$ and $K_2{'}$ the Koszul
complexes associated to $S$ and $\lambda$ and $T{'}$ and $\mu$
respectively, then the total complex of the double complex 
obtained from the tensor product of $K_1$ and $K_2{'}$, 
$K_1\tilde{\otimes}K_2{'}$ is isomorphic to $\tilde
{K}$, the Koszul complex associated to $(S\otimes I ,
I\otimes T{'})$ and $(\lambda ,\mu)$ on $X_1\tilde{\otimes }X_2$, 
which is naturally isomorphic
to the Koszul complex of $(L_S,R_T)$ and $(\lambda ,\mu)$ on
${\rm L}(J)$, see [5, section 3].\par

\indent In order to state our description of the semi-Fredholm and the
semi-Browder joint spectra of the tuple $(L_S, R_T)$, 
as we have done in section 4, we first
describe the defect and the approximate point spectra of the 
mentioned tuple.\par

\begin{pro} Let $X_1$ and $X_2$ be two complex Banach
spaces, and $J$ and operator ideal between
$X_2$ and $X_1$ in the sense of [5]. Let us
consider two tuples of commuting operators defined on $X_1$ and 
$X_2$, $S$ and $T$ respectively. Then, if $(L_S,R_T)$ 
is the tuple of left and right multiplications defined on $L(J)$, we have that
\par
\noindent  {\rm (i)} $\sigma_{\delta}(S)\times \sigma_{\pi }(T)\subseteq
\sigma_{\delta}(L_S,R_T)\subseteq sp_{\delta}
(L_S,R_T)
 \subseteq sp_{\delta}(S)
\times sp_{\pi}(T),$\par
\noindent  {\rm (ii)} $
\sigma_{\pi}(S)\times \sigma_{\delta }(T)\subseteq
\sigma_{\pi}(L_S,R_T) \subseteq sp_{\pi }
(L_S,R_T)\subseteq sp_{\pi}(S)\times sp_{\delta}(T).
$\par
\noindent In addition, if $X_1$ and $X_2$ are Hilbert spaces, the above inclusions
are equalities.
\end{pro}

\begin{proof}
 
\indent As we have said, $J$ may be seen as the tensor product of $X_1$
and $X_2{'}$, $X_1\tilde{\otimes}X_2{'}$, relative to $<X_1,X_1{'}>$
and $<X_2,X_2{'}>$, and $(L_S,R_T)$ may be identified with the
tuple $(S\otimes I ,I\otimes T{'})$. Moreover, if ${\mathcal K}_1 
$ and ${\mathcal K}_2{'}$ denote the differential space 
associated to $K_1$
and $K_2{'}$ respectively, then $\tilde{\mathcal K}$,
the differentiable space associated to $\tilde{K}$, is
isomorphic to ${\mathcal K}_1\tilde{\otimes}{\mathcal K}_2{'}$ 
(see [5, section 3]).\par
   
 \indent In addition, since for all $i=1,\ldots ,n$ $S_i\in {\rm L}(X_1)$
and for all $j=1,\ldots ,m$ $T_j\in {\rm L}(X_2)$, the differential
spaces ${\mathcal K}_1 $ and ${\mathcal K}_2{'}$ 
satisfy the conditions of [5, Theorem 2.2], and by means of an argument similar
to the one of Proposition 4.1 we have that
$$
\sigma_{\delta}(S)\times \sigma_{\delta }(T{'})\subseteq
\sigma_{\delta}(S\otimes I ,I\otimes T{'})=
\sigma_{\delta}(L_S,R_T).
$$
\indent However, by [8, Theorem 2.0], $\sigma_{\pi}(T)=\sigma_{\delta}
(T{'})$. Thus, we have proved the left hand side inclusion of the first
statement.\par

\indent The middle inclusion is clear.\par    

\indent In order to see the right hand inclusion, let us first observe
that if $\mu\notin sp_{\pi}(T)$, then $\mu\notin sp_{\delta}(T{'})$.
\par

\indent In fact, if $K_2$ is split at level $p=m$, then by
[8, Lemma 2.2] $K_2{'}$ is split at level $p=0$.\par

\indent Now, by the isomorphism of [8, Lemma 2.2], if we think the homotopy 
operator which gives the
splitting for the complex $K_2{'}$ at level $p=0$ as a matrix, then each
component of the matrix is an adjoint operator. In particular, 
by means of the properties of the tensor product of [5], it is possible to 
adapt the proof of the corresponding inclusion of
Proposition 4.1 in order to see that if $(\lambda ,\mu)\notin sp_{\delta}(S)
\times sp_{\pi}(T)$, then $(\lambda ,\mu)\notin sp_{\delta}(S\otimes
I ,I\otimes T{'})=sp_{\delta}(L_S ,R_T)$.\par    

\indent The second statement may be proved by means of a similar
argument.
\end{proof}

\indent In the following proposition we give our description of the 
semi-Fredholm joint spectra of the tuple $(L_S,R_T)$.\par

\begin{pro} Let $X_1$ and $X_2$ be two complex Banach
spaces, and $J$ and operator ideal between
$X_2$ and $X_1$ in the sense of [5]. Let us
consider two tuples of commuting operators defined on $X_1$ and 
$X_2$, $S$ and $T$ respectively. Then, if $(L_S,R_T)$ 
is the tuple of left and right multiplications defined on $L(J)$, 
we have that

\begin{align*}
{\rm (i)}\hbox{  }&\sigma_{\Phi_{-}}(S)\times \sigma_{\pi }(T)\cup 
\sigma_{\delta}(S)\times\sigma_{\Phi_{+}}(T)\subseteq
\sigma_{\Phi_{-}}(L_S,R_T)\subseteq\\
& sp_{\delta e}(L_S,R_T)
\subseteq sp_{\delta e }(S)\times sp_{\pi}(T)\cup
sp_{\delta  }(S)\times sp_{\pi e}(T),\\\end{align*}

\begin{align*}
{\rm (ii)}\hbox{  }&\sigma_{\Phi_{+}}(S)\times \sigma_{\delta }(T)\cup 
\sigma_{\pi}(S)\times\sigma_{\Phi_{-}}(T)\subseteq
\sigma_{\Phi_{+}}(L_S,R_T)\subseteq\\
& sp_{\pi e}(L_S,R_T)
\subseteq sp_{\pi e }(S)\times sp_{\delta}(T)\cup
sp_{\pi  }(S)\times sp_{\delta e}(T).\\\end{align*}

\noindent In addition, if $X_1$ and $X_2$ are Hilbert spaces,
the above inclusions are equalities.
\end{pro}

\begin{proof}

\indent By means of an argument similar to the one of Proposition 4.2,
adapted as we have done in Poposition 6.1, it is possible to see that
$$
\sigma_{\Phi_{-}}(S)\times \sigma_{\delta }(T{'})\cup 
\sigma_{\delta}(S)\times\sigma_{\Phi_{-}}(T{'})\subseteq
\sigma_{\Phi_{-}}(S\otimes I,I\otimes T{'})=
\sigma_{\Phi_{-}}(L_S,R_T).
$$
\indent However, by [8, Theorem 2.0] $\sigma_{\delta }(T{'})=\sigma_{\pi}(T)$,
and by elementary properties of the adjoint of an operator it is easy
to see that $\sigma_{\Phi_{+}}(T)\subseteq \sigma_{\Phi_{-}}
(T{'})$. Thus, we have seen the left hand side inclusion of the 
first statement.\par

\indent The middle inclusion is clear.\par 

\indent Let us consider $(\lambda ,\mu)\in sp_{\delta e}
(L_S,R_T)\setminus (sp_{\delta e }(S)\times sp_{\pi}(T)\cup
sp_{\delta  }(S)\times sp_{\pi e}(T))$. By Proposition 6.1 we
have that $\lambda\in sp_{\delta}(S)\setminus sp_{\delta e}(S)$
and $\mu\in sp_{\pi}(T)\setminus sp_{\pi e}(T)$. However, by
[8, Lemma 2.2] and elementary properties of the adjoint of an operator
we have that $\mu\in sp_{\delta}(T{'})\setminus sp_{\delta e}(T{'})$.
Then, as in Proposition 4.2, there are two linear bounded maps
$h\colon X_1\to X_1\otimes\wedge^1\mathbb{C}^n$, $g{'}\colon X_2{'}\to 
X_2{'}\otimes\wedge^1\mathbb{C}^m$, and two finite rank projectors
$k_1\colon X_1\to X_1$ and $k_2{'}\colon X_2\to X_2$ such that
$$
d_{11}\circ h=I-k_1,\hskip.5cm d_{21}{'}\circ g{'}=
I-k_2{'},
$$  
where $d_{21}{'}$ is the boundary map of the complex $K_2{'}$ 
at level $p=1$.\par

\indent Now, by the isomorphism of [8, Lemma 2.2], if we think the map $g{'}$
as a matrix, then each component of the matrix is an adjoint operator.
Then, by the properties of the tensor product introduced
in [5], we may consider the well defined map 
$$
H\colon X_1\tilde{\otimes} X_2{'} \to (K_1\tilde{\otimes}
K_2{'})_1,\hskip.5cm H=(h\otimes I,I\otimes g{'}).
$$
Now, by an argument similar to the one of Proposition 4.2, it is easy
to see that $(\lambda ,\mu)\notin sp_{\delta e}(S\otimes I,
I\otimes T{'} )=sp_{\delta e}(L_S ,R_T)$,
which is impossible by our assumptions.\par

\indent By means of a similar argument it is possible to see the second
statement.
\end{proof}

\indent We now give our description of the semi-Browder joint spectra
of the tuple of left and right multiplications $(L_S ,R_T)$ defined
on ${\rm L}(J)$.\par

\begin{thm} Let $X_1$ and $X_2$ be two complex Banach
spaces, and $J$ and operator ideal between
$X_2$ and $X_1$ in the sense of [5]. Let us
consider two tuples of commuting operators defined on $X_1$ and 
$X_2$, $S$ and $T$ respectively. Then, if $(L_S,R_T)$ 
is the tuple of left and right multiplications defined on $L(J)$, we have that

\begin{align*}
{\rm (i)}\hbox{  }&\sigma_{{\mathcal B}_{-}}(S)\times \sigma_{\pi }(T)\cup 
\sigma_{\delta}(S)\times\sigma_{{\mathcal B}_{+}}(T)\subseteq
\sigma_{{\mathcal B}_{-}}(L_S,R_T)\subseteq\\
& sp_{{\mathcal B}_{-}}(L_S,R_T)
\subseteq sp_{{\mathcal B}_{-} }(S)\times sp_{\pi}(T)\cup
sp_{\delta  }(S)\times sp_{{\mathcal B}_{+}}(T),\\\end{align*}

\begin{align*}
{\rm (ii)}\hbox{ }&\sigma_{{\mathcal B}_{+}}(S)\times \sigma_{\delta }(T)\cup 
\sigma_{\pi}(S)\times\sigma_{{\mathcal B}_{-}}(T)\subseteq
\sigma_{{\mathcal B}_{+}}(L_S,R_T)\subseteq\\
& sp_{{\mathcal B}_{+}}(L_S,R_T)
\subseteq sp_{{\mathcal B}_{+} }(S)\times sp_{\delta}(T)\cup
sp_{\pi  }(S)\times sp_{{\mathcal B}_{-}}(T),\\\end{align*}

\noindent In addition, if $X_1$ and $X_2$ are Hilbert spaces,
the above inclusions are equalities.
\end{thm}
\begin{proof} 
\indent In order to see the first statement, let us observe that if $K_1^r$ is the Koszul
complex associated to the  tuple $(S-\lambda)^r=((S_1-\lambda_1)^r,
\ldots ,(S_n-\lambda_n)^r)$, and if $K_2{'}^r$ is the Koszul
complex associated to the tuple $(T{'}-\mu)^r=((T{'}_1-\mu_1)^r,
\ldots ,(T_m-\mu_m)^r)$, then $\tilde{K}^r$, the
Koszul complex associated to the tuple $((S-\lambda)^r\otimes T,
I\otimes (T{'}-\mu)^r)$, is isomorphic to the total complex
obtained from the double complex of the tensor product
of $K_1^r$ and $K_2{'}^r$, i.e., $\tilde{K}^r\cong 
K_1^r\tilde{\otimes}K_2{'}^r$ (see [5, section 3]). \par

\indent Now, we may adapt the proof of the left hand inclusion of
Theorem 5.1, as we have done in Proposition 6.1, using in particular
Proposition 6.2 instead of Proposition 4.2, in order to see that 
$\sigma_{{\mathcal B}_{-}}(S)\times 
\sigma_{\delta }(T{'})\subseteq\sigma_{{\mathcal B}_{-}}(S\otimes I,
I\otimes T{'})=
\sigma_{{\mathcal B}_{-}}(L_S,R_T)$. However, 
by [8, Theorem 2.0], $\sigma_{\pi}(T)=\sigma_{\delta}(T{'})$. Thus,
$\sigma_{{\mathcal B}_{-}}(S)\times 
\sigma_{\pi }(T)\subseteq
\sigma_{{\mathcal B}_{-}}(L_S,R_T)$
\par
\indent A similar argument, using in particular that
$\sigma_{{\mathcal B}_{-}}(T{'})=\sigma_{{\mathcal B}_{+}}(T)$ (see [7, Theorem 11]),
gives us that $\sigma_{\delta}(S)\times\sigma_{{\mathcal B}_{+}}(T)\subseteq
\sigma_{{\mathcal B}_{-}}(L_S,R_T)$.\par

\indent The middle inclusion is clear.\par

\indent In order to see the right hand inclusion, it is possible to adapt 
the proof of the corresponding part of Theorem 5.1. \par

\indent Indeed, if we use Proposition 6.2  instead of Proposition 4.2, we 
have that $sp_{\delta e}(L_S, R_T)\subseteq  sp_{{\mathcal B}_{-} }(S)\times sp_{\pi}(T)\cup
sp_{\delta  }(S)\times sp_{{\mathcal B}_{+}}(T)$. On the other hand,
if we suppose that $(\lambda ,\mu)\in \tilde{\mathcal A}(L_S, R_T)\setminus ( sp_{{\mathcal B}_{-} }(S)\times sp_{\pi}(T)\cup
sp_{\delta  }(S)\times sp_{{\mathcal B}_{+}}(T))$, then it is possible
to adapt the argument of Theorem 5.1 in order to get a contradiction.
However, in order to adapt this part of the proof, we have to observe 
the following facts.\par

\indent First, by [8, Lemma 2.2], if $\mu\notin sp_{\pi e}(T)$, then $\mu
\notin sp_{\delta e}(T{'})$. Moreover, if there exists $l\in\mathbb{N}$
such that for all $r\ge l$ $\dim H_m(K_2^r)=\dim H_m
(K_2^l)$, then by [7, Theorem 11] it is easy to see that
$\dim H_0(K_2{'}^r)=\dim H_0(K_2{'}^l)$, for all $r\ge l$. In
addition, if $\mu \notin sp_{\delta e}(T{'})$, by the analytic
spectral mapping theorem for the essential split defect spectrum, 
the complex $K_2{'}^r$ are Fredholm split for all $r\in \mathbb{N}$,
i.e., there are operators $g_r{'}\colon X_2{'}\to X_2{'}\otimes\wedge^1
\mathbb{C}^m$ and finite rank projectors $k_{2r}{'}\colon X_2{'}\to 
X_2{'}$ such that 
$d_{21}^r{'}\circ g_r{'}=I-k_{2r}{'}$, where $d_{21}^r{'}$ denotes
the chain map of the complex $K_2{'}^r$ at level $p=1$. Furthermore,
by [8, Lemma 2.2], if for $r\in \mathbb{N}$ we think the map $g_r{'}$ as matrix,
then each component of the matrix is an adjoint operator, and by 
elementary properties of the adjoint of an operator, the maps 
$g_r{'}$ and $k_{2r}{'}$ may be chosen with the same properties
of the maps $g_r$ and $k_{2r} $ of Theorem 5.1. With all this observations
it is possible to conclude the proof of the right hand side inclusion of the first
statement.\par

\indent The second statement may be proved by means of a similar
argument.
\end{proof}

\bibliographystyle{amsplain}

\begin{thebibliography}{99}

\bibitem{ }J. J. Buoni, R. Harte and T. Wickstead, Upper and lower Fredholm
spectra, Proc. Amer. Math. Soc.  66 (1977), 309-314.

\bibitem{ }R. E. Curto and A. T. Dash, Browder spectral systems, Proc. Amer. Math.
Soc. 103  (1988), 407-413.  

\bibitem{ } A. T. Dash, Joint Browder spectra and tensor products,
Bull. Austral. Math. Soc. 32 (1985), 119-128.

\bibitem{ } J. Eschmeier, Analytic spectral mapping theorems for joint spectra,
Oper. Theory. Adv.  Appl. 24 (1987), 167-181.

\bibitem{ } J. Eschmeier, Tensor products and elementary operators,
J. Reine Angew. Math. 390 (1988), 47-66.

\bibitem{ }T. Ichinose, Spectral properties of tensor prodcut of linear operators I,
Trans. Amer. Math. Soc. 235 (1978), 75-113.

\bibitem{ } V. Kordula, V. M\"uller and V. Rako\v cevi\'c,
On the semi-Browder spectrum, Studia Math. 123 (1997), 1-13. 

\bibitem{ } Z. S\l odkowski, An infinite family of joint spectra,  Studia Math. 61 (1977), 239-255.

\bibitem{ }J. L. Taylor, A joint spectrum for several commuting   
operators, J. Funct. Anal. 6 (1970), 172-191.
\end{thebibliography}

\vskip.5cm

\noindent Enrico Boasso\par
\noindent E-mail address: enrico\_odisseo@yahoo.it

\end{document}